\documentclass[a4paper,11pt]{amsart}

\usepackage{amsmath,amssymb,amsthm}
\usepackage{hyperref}
\allowdisplaybreaks %??????

\numberwithin{equation}{section} \theoremstyle{plain}
\newtheorem{thm}{Theorem}[section]

\newtheorem{prop}[thm]{Proposition}
\newtheorem{lem}[thm]{Lemma}
\newtheorem{cor}[thm]{Corollary}
\newtheorem{defn}[thm]{Definition}
\newtheorem{conj}[thm]{Conjecture}

\newtheorem{rem}[thm]{Remark}

\newtheorem*{acknow}{Acknowledgments}
   
 \makeatletter

\newtheorem{Weak Chern Conjectures}[thm]{Weak Chern Conjectures}
\newtheorem*{Weak Chern Conjectures0}{Weak Chern Conjectures}
\newtheorem{Chern Conjecture}[thm]{Chern Conjectures}
\newtheorem*{Chern Conjecture0}{Chern Conjecture}

\def\<{\langle}
\def\>{\rangle}
\def\({\left(}
\def\){\right)}
\def\[{\left[}
\def\]{\right]}

\allowdisplaybreaks %??????
\makeatother
\title {Integral-Einstein  hypersurfaces in spheres}%{Some applications of  height function  for minimal hypersurfaces in spheres}

\author[J.Q. Ge]{Jianquan Ge${}^{1}$}
\address{${}^{1}$School of Mathematical Sciences, Laboratory of Mathematics and Complex Systems, Beijing Normal University, Beijing 100875, P.R. CHINA.}
\email{jqge@bnu.edu.cn}

\author[F.G. Li]{Fagui Li${}^{2,*}$}
\address{$^{2}$Frontier Interdisciplinary Domain, Beijing Institute of Technology, Zhuhai, Guangdong 519088, P. R. CHINA.}
\email{lifagui@bitzh.edu.cn}

\date{}
\thanks {$^{*}$ the corresponding author.}
\thanks{J. Q. Ge is partially supported by NSFC (No. 12171037) and the Fundamental Research Funds for the Central Universities.}
\thanks{F. G. Li is partially supported by NSFC (No. 12171037, 12271040).}
\begin{document}
\maketitle

% % % % % % % % % %
\begin{abstract}
Combining the intrinsic and extrinsic geometry, we generalize Einstein manifolds to Integral-Einstein (IE) submanifolds. A Takahashi-type theorem is established to characterize minimal hypersurfaces with constant scalar curvature (CSC) in unit spheres which are conjectured to be isoparametric in the Chern conjecture. For these hypersurfaces, we obtain some integral inequalities with the bounds characterizing exactly the totally geodesic hypersphere, the non-IE minimal Clifford torus $S^{1}(\sqrt{\frac{1}{n}})\times S^{n-1}(\sqrt{\frac{n-1}{n}})$ and the IE minimal CSC hypersurfaces. Moreover, if further the third mean curvature is constant, then it is an IE hypersurface or an isoparametric hypersurface with $g\leq2$ principal curvatures. In particular, all minimal isoparametric hypersurfaces with $g\geq3$ principal curvatures are IE hypersurfaces. As applications, we obtain some spherical Bernstein theorems, including that any embedded closed minimal surface of genus no more than $\mathfrak{g}$ inside a tubular neighborhood of constant radius $r(\mathfrak{g})>0$ around an equator in $\mathbb{S}^3$ is an equator.
\end{abstract}

%%%%%%%%%%%%%%%%%%%%%%%
\section{Introduction}
In 1969, Lawson \cite{Lawson 1969} gave a classification of minimal Einstein hypersurfaces in unit spheres, i.e., \emph{if $M^n\subset \mathbb{S}^{n+1}$ is Einstein, then either it is   totally geodesic, or $n =2k$ and it is an open submanifold of}
$$
M_{k,k}=
S^{k}(\sqrt{\frac{1}{2}})\times S^{k}(\sqrt{\frac{1}{2}})
\subset \mathbb{S}^{n+1}.
$$
Meanwhile, Ryan \cite{Ryan 1969} classified Einstein hypersurfaces in all space forms without  the minimal condition.  In particular, \emph{if $M^n\subset \mathbb{S}^{n+1}$ $\left(n\geq 3 \right) $ is a closed  Einstein hypersurfaces, then $M^n$ is either a totally umbilical hypersphere,  or  one of  $S^{k}(\sqrt{\frac{k-1}{n-2}})\times S^{n-k}(\sqrt{\frac{n-k-1}{n-2}})$,
$\left( 2 \leq k \leq n-2\right) $}.
%(see more generalizations in \cite{Onti 2017}).
These Einstein hypersurfaces only consist of isoparametric hypersurfaces with no more than $2$ principal curvatures (except $S^{1}(r)\times S^{n-1}(t)$,  $t=\sqrt{1-r^2}$ and $0<r<1$) in $\mathbb{S}^{n+1}$. Recall that isoparametric hypersurfaces in unit spheres are hypersurfaces which have $g\in\{1,2,3,4,6\}$ distinct constant principal curvatures. The classification problem was studied extensively, since late 1930s initiated by Cartan (for $g\leq3$) till to the year 2020 completed by Miyaoka \cite{Miy13} (for $g=6$) and finally by Cecil, Jenson and  Chi \cite{CCJ07, Chi11, Chi13} (for $g=4$) (see a number of important contributions in references of the excellent book \cite{CR15} and the elegant survey \cite{Chi19}). In this paper, by combining the intrinsic and extrinsic geometry we introduce an extension of Einstein hypersurfaces so as to include these fascinating hypersurfaces.

A Riemannian manifold $(M^n, \mathbf{g})$ $(n\geq3)$ is called Einstein if it satisfies the pointwise intrinsic condition (cf. \cite{Besse 1987}) $$ {\rm Ric}=\frac{R}{n} \mathbf{g},$$
where $\rm Ric$ is the Ricci curvature tensor, and $R$ is the scalar curvature which is constant by Schur's theorem.
By the famous Nash embedding theorem \cite{John Nash 1954}, $(M^n, \mathbf{g})$ is always realizable as a submanifold of a Euclidean space $\mathbb{R}^{N}$.
To relax the pointwise intrinsic condition of Einstein manifolds, we restrict the extrinsic geometry of the submanifold by taking an integral as follows.
\begin{defn}\label{defintion integral Eins manifold}
Let $M^n$ $(n\geq3)$ be a compact submanifold in the Euclidean space $\mathbb{R}^{N}$. We call $M^n$ an Integral-Einstein (IE) submanifold if for any unit vector $a\in \mathbb{S}^{N-1}$,
\begin{equation}\label{equation integral Einstein manifold}
\int_{M}\left( {\rm Ric}-\frac{R}{n}\mathbf{g}\right)
(a^{\rm T},a^{\rm T})=0,
\end{equation}
 where $a^{\rm T}\in\Gamma(TM)$ denotes the tangent component of the constant vector $a$ along $M^n$.
\end{defn}
For noncompact submanifolds, one can also define the IE property by requiring the integral equation (\ref{equation integral Einstein manifold}) over any geodesic ball $B_R(p)$ of sufficiently large radius $R$,
or over certain exhausting compact domains (which might be useful for those noncompact manifolds as total spaces of vector bundles).

 It is a natural problem whether there is a Nash-type embedding theorem for IE submanifolds, i.e., can any Riemannian manifold be embedded as an IE submanifold in a Euclidean space $\mathbb{R}^{N}$? More discreetly, one should allow the embedding to be IE up to some ambient transformations like Lie sphere transformations (cf. \cite{Ce08}) which include spherical parallel translations for submanifolds of spheres. It turns out that all isoparametric hypersurfaces of $\mathbb{S}^{n+1}$ except $S^{1}(r)\times S^{n-1}(\sqrt{1-r^2})$ are IE hypersurfaces (up to spherical parallel translations), including those non-Einstein minimal isoparametric hypersurfaces with $g\geq3$ principal curvatures (see Corollary \ref{cor-f3-isop} and Theorem \ref{thm Volume estimation of minimal isoparametric hypersurfaces}). The only left case $S^{1}(r)\times S^{n-1}(\sqrt{1-r^2})$ provides a candidate of counterexample to the IE embedding problem above, since we do not know wether it can be embedded as a higher codimensional IE submanifold.

Another motivation comes from the study of the Chern Conjecture \cite{SWY12}:
% which asserts that \emph{a closed minimal hypersurface $M^n$ with constant scalar curvature (CSC) in $\mathbb{S}^{n+1}$ is isoparametric.} 
\begin{Chern Conjecture0}\cite{SWY12}
A closed minimal hypersurface $M^n$ with constant scalar curvature (CSC) in $\mathbb{S}^{n+1}$ is isoparametric.
\end{Chern Conjecture0}
 By the Simons inequality and the pinching rigidity \cite{CDK70, Lawson 1969,Simons68}, $M^n$ is either totally geodesic or a Clifford torus $S^{k}(\sqrt{\frac{k}{n}})\times S^{n-k}(\sqrt{\frac{n-k}{n}})$, which are isoparametric with $g\leq2$, if the constant squared length of the second fundamental form $S:=\|A\|^2\leq n$. Hence, only the case of $S>n$ is left to be verified as minimal isoparametric hypersurfaces with $g\geq3$ principal curvatures which have constant $S=(g-1)n>n$. The first nontrivial case when $n=3$ was proven by Chang \cite{Chang93}, while for higher dimensions it is still open in general (see various partial results in \cite{DX11,PT2,SWY12}, and see a recent important progress in \cite{TWY18, TY20} which generalized the $3$-dimensional result of \cite{AB90} to all dimensions).

 During the study of such minimal CSC hypersurfaces, we find that the following position and normal position height functions $\varphi_a(x), \psi_a (x)$ take important roles as in Minkowski's integral formula (cf. \cite{Reilly 1970}). For any unit vector $a \in \mathbb{S}^{n+1}$,  the height functions are defined as
\begin{equation}\label{height functions}
\varphi_a(x) = \langle x,a \rangle,
\quad  \psi_a (x)=\langle \nu,a \rangle,
\end{equation}
where $\nu$ is the unit normal vector field along $x\in M^n$.
There are many applications  of minimal submanifolds in spheres by using the height functions recently, such as Solomon-Yau's conjecture %the volume gap
 \cite{Ge Li 2022 Volume gap}, Perdomo's conjecture \cite{Ge Li 2020 A lower bound second fundamental form} and  isoperimetric-type inequality \cite{Li-Chen isoperimetric inequality}.
The well known Takahashi theorem \cite{Takahashi 1966} states that \emph{$M^n$ is minimal if and only if there exists a constant $\lambda$ such that $\Delta \varphi_a =-\lambda\varphi_a$
for all $a \in \mathbb{S}^{n+1}$}. Analogously, we find that the same equation for $\psi_a$ is a sufficient and necessary condition for minimal CSC hypersurfaces (see Theorem \ref{thm the equivalent of minimal  CSC CMC hypersurfaces}). Similar characterization for constant mean curvature is also obtained. These lead us to study the uniform bounds of the $L^2$ squared norm of the position height function $\varphi_a$ on minimal CSC hypersurfaces. It turns out that the bounds characterize exactly the totally geodesic hypersphere, the non-IE minimal Clifford torus $S^{1}(\sqrt{\frac{1}{n}})\times S^{n-1}(\sqrt{\frac{n-1}{n}})$ and the IE minimal CSC hypersurfaces (see Theorem \ref{thm main Volume estimation of minimal hypersurface}). Here the Reilly formula (\cite{Reilly 1977}) is applied and then an alternative characterization of IE hypersurfaces follows (see Theorem \ref{thm-IE-hypers-charact}), which shows the non-IE property of $S^{1}(r)\times S^{n-1}(\sqrt{1-r^2})\subset\mathbb{S}^{n+1}$ and the IE property of minimal isoparametric hypersurfaces with $g\geq3$ principal curvatures. Moreover, if $ M^n$ is a closed minimal CSC hypersurface in $\mathbb{S}^{n+1}$  with $S=\|A\|^2>n$ and constant $f_3={\rm Tr}(A^3)$, then $ M^n$ is an IE hypersurface (see Corollary \ref{cor-f3-isop}). Here the Cheng-Yau operator \cite{Cheng Yau Math Ann 1977} is applied which greatly simplifies the proof for isoparametric hypersurfaces in Theorem \ref{thm Volume estimation of minimal isoparametric hypersurfaces}. As applications of the integral inequalities about $\varphi_a$, we also obtain some spherical Bernstein theorems (see Theorem \ref{thm  applitions of intergal formula}). Specifically, we demonstrate that the non-totally geodesic minimal hypersurfaces  cannot curl up near an equator.
In particular,  if $n=2$,  the distance from the equator is only related to the  Euler characteristic $\chi$ of  the minimal surface $M^2$. Hence any embedded closed minimal surface of genus no more than $\mathfrak{g}$ inside a tubular neighborhood of constant radius $r(\mathfrak{g})>0$ around an equator in $\mathbb{S}^3$ is an equator. This can be compared to the result that there are infinitely many minimal surfaces in some neighborhood of an equator (cf. \cite{Kapouleas and Yang 2010,Wiygul  David 2020}).
  %the positive constant $C$  depends only  on  Euler characteristic $\chi$ of $M^2$.

\section{Main results}\label{sec-main-results}
Firstly, we give a generalization of the classical Takahashi theorem \cite{Takahashi 1966} (case (i) and $H=0$ in
Theorem \ref{thm the equivalent of minimal  CSC CMC hypersurfaces}) for hypersurfaces in unit spheres.
\begin{thm}\label{thm the equivalent of minimal  CSC CMC hypersurfaces}
Let $M^n$ be  a  connected hypersurface immersed in $\mathbb{S}^{n+1}$ with mean curvature $H:={\rm Tr}(A)/n$ and
squared length of the second fundamental form $S:=\|A\|^2$.
\begin{itemize}
\item [(i)]  $H$ is constant
if and only if
there exist some continuous function $\lambda$ and constant $\mu$ such that for all $a \in \mathbb{S}^{n+1}$,
$$
\Delta \varphi_a =-\lambda\varphi_a +n\mu\psi_a,
$$
in which case $\lambda=n$ and $\mu=H$. % are constant. 
In particular, $H=0$
if and only if
there exists a continuous function $\lambda$ such that $\Delta \varphi_a =-\lambda\varphi_a$
for all $a \in \mathbb{S}^{n+1}$.
\item [(ii)]  $H$ is constant
if and only if
there exist some continuous functions $\lambda$ and $\mu$ such that for all $a \in \mathbb{S}^{n+1}$,
$$
\Delta \psi_a =-\lambda\psi_a +n\mu\varphi_a,
$$
in which case $\lambda=S$ and $\mu=H$. % is constant. 
In particular, $H=0$
if and only if
there exists a continous function  $\lambda$  such that $\Delta \psi_a =-\lambda\psi_a$ for all $a \in \mathbb{S}^{n+1}$.
\item [(iii)]  $H$ and $S$ are both constant
if and only if
there exist some constant $\lambda$ and  continuous function $\mu$ such that for all $a \in \mathbb{S}^{n+1}$,
$$
\Delta \psi_a =-\lambda\psi_a +n\mu\varphi_a,
$$
in which case $\lambda=S$ and $\mu=H$. % are constant. 
In particular, $H=0$ and $S={\rm Constant}$
if and only if
there exists a constant $\lambda$  such that
$\Delta \psi_a =-\lambda\psi_a$
for all $a \in \mathbb{S}^{n+1}$.
\end{itemize}
\end{thm}

Next we give a characterization of IE hypersurfaces in unit spheres.
\begin{thm}\label{thm-IE-hypers-charact}
Let $M^n$ be a closed  hypersurface immersed in $\mathbb{S}^{n+1}$. Then \emph{$M^n$} is IE if and only if for all  $a \in \mathbb{S}^{n+1}$,
\begin{equation}\label{equation integral Einstein  hypersurface}
\int_{M}\left( 1-(n+1)\varphi_a^2-\psi_a^2\right)  =
\int_{M}\left( \rho-1\right)\left( 1-\varphi_a^2-(n+1)\psi_a^2\right),
\end{equation}
where $\rho-1=\frac{n^2H^2-S}{n(n-1)}$ and $\rho=\frac{R}{n(n-1)}$ is the normalized scalar curvature.
In particular, we have  the following special cases.
\begin{itemize}
\item[(A)] If $M^n$ is minimal, then \emph{$M^n$} is IE if and only if
$$
 \int_{M}S
\left(  1-\varphi_a^2-(n+1)\psi_a^2\right) =0, \quad \textit{for all } a \in \mathbb{S}^{n+1}.
$$
\item[(B)]
If $M^n$ is minimal and $S>0$ is constant, then we have
\begin{equation}\label{IEintegral-MCSC}
\int_{M}\left( {\rm Ric}-\frac{R}{n}\mathbf{g}\right)
(a^{\rm T},a^{\rm T})=S\Big((n+2) \int_{M}{\varphi_a^2}- {\rm Vol}(M^n)\Big).
\end{equation}
In this case, \emph{$M^n$} is IE if and only if any one of the follows holds:
\begin{itemize}
\item[(i)]
\begin{equation}\label{equation integral Einstein  hypersurface H=0 S=constant}
\int_{M}\varphi_a^2=\frac{1}{n+2}{\rm Vol }(M^n), \quad \textit{for all } a \in \mathbb{S}^{n+1};
\end{equation}
\item[(ii)]
$$
\int_{M}\psi_a^2=\frac{1}{n+2}{\rm Vol }(M^n), \quad \textit{for all } a \in \mathbb{S}^{n+1};
$$
\item[(iii)]
$$
\int_{M}\varphi_a^2=\int_{M}\psi_a^2, \quad \textit{for all } a \in \mathbb{S}^{n+1};
$$
\item[(iv)]
$$
\int_{M}\varphi_a \psi_a f_3=0, \quad \textit{for all } a \in \mathbb{S}^{n+1},
$$
 where $f_3={\rm Tr}(A^3)=3\binom{n}{3}H_3$ and $H_3$ is the third mean curvature.
\end{itemize}
\end{itemize}
\end{thm}

\begin{rem}
For $n=2$, the equation $(\ref{equation integral Einstein manifold})$ for the definition of IE submanifolds is automatically satisfied and so is the integral formula $(\ref{equation integral Einstein  hypersurface})$, which is nontrivial and new to our best knowledge. Notice that $(\ref{equation integral Einstein  hypersurface})$ can be rewritten as
\begin{equation*}
\int_{M}\left( I_{n+2}-(n+1)xx^t-\nu\nu^t\right) =
\int_{M}\left( \rho-1\right) \left( I_{n+2}-xx^t-(n+1)\nu\nu^t\right),
\end{equation*}
where $I_{n+2}$ is the identity matrix, $xx^t$ and $\nu\nu^t$ are regarded as matrix-valued functions.
\end{rem}

\begin{cor}\label{cor-f3-isop}
A closed minimal CSC hypersurface in $\mathbb{S}^{n+1}$  with $S>n$ and constant third mean curvature is an IE hypersurface.
In particular, every minimal isoparametric hypersurface with $g\geq3$ principal curvatures in $\mathbb{S}^{n+1}$ is an IE hypersurface. Moreover, the Clifford torus $S^{1}(r)\times S^{n-1}(\sqrt{1-r^2})\subset\mathbb{S}^{n+1}$ $(0<r<1)$ is not IE.
\end{cor}

%\begin{rem}
%It is natural to ask a weak Chern conjecture that a closed minimal CSC hypersurface in $\mathbb{S}^{n+1}$ with $S>n$ must be an IE hypersurface.
%\end{rem}

\begin{rem}
Both of IE minimal CSC hypersurfaces and minimal isoparametric hypersurfaces with $g\geq3$ share the same average-symmetric property $(\ref{equation integral Einstein  hypersurface H=0 S=constant})$, namely, the $L^2$ squared norm of any coordinate function equals the average $\frac{1}{n+2}{\rm Vol }(M^n)$.
\end{rem}
It is natural to ask the following Weak Chern Conjectures:
\begin{conj}[Weak Chern Conjectures]\label{Weak Chern Conjectureskey}
%\begin{Weak Chern Conjectures0}
\ 
\begin{itemize}
\item[(i)]  
 A closed minimal CSC hypersurface in $\mathbb{S}^{n+1}$ with $S > n$ is an IE hypersurface.
\item[(ii)]    A closed minimal CSC IE hypersurface in $\mathbb{S}^{n+1}$ with $S > n$ is isoparametric.
\end{itemize}
%\end{Weak Chern Conjectures0}
\end{conj}
By Corollary \ref{cor-f3-isop}, we  obtain the following equivalence relationship between the Chern Conjecture and the Weak Chern Conjectures.
\begin{prop}
The validity of the Chern Conjecture  is equivalent to the simultaneous validity of Conjectures \ref{Weak Chern Conjectureskey} (i)-(ii).
%Conjectures \ref{Weak Chern Conjectureskey} (i)-(ii)  hold  if and only if the  Chern Conjecture is true. 
\end{prop}

The following inequalities imply a sharp gap as the Simons inequality mentioned before for minimal CSC hypersurfaces in $\mathbb{S}^{n+1}$. In particular, the equality cases characterize exactly the totally geodesic hypersphere, the IE minimal CSC hypersurfaces and the non-IE minimal Clifford torus $S^{1}(\sqrt{\frac{1}{n}})\times S^{n-1}(\sqrt{\frac{n-1}{n}})$ (see other characterizations of this Clifford torus in \cite{Cheng Qing-Ming 1996}, etc).
\begin{thm}\label{thm main Volume estimation of minimal hypersurface}
Let $ M^n$ be a closed minimal hypersurface  immersed  in $\mathbb{S}^{n+1}$. Then
\begin{equation}\label{equation  main Volume estimation of minimal hypersurface}
0\leq
\inf_{a\in \mathbb{S}^{n+1}}
\frac {\int_{M} \varphi^2_a}
{  {\rm Vol }\left( M^n \right)  }
\leq\frac{1}{n+2}
\leq\sup_{a\in \mathbb{S}^{n+1}}
\frac{\int_{M} \varphi^2_a}
{ {\rm Vol }\left( M^n \right)  }\leq
\frac{1}{n+1}.
\end{equation}
%where $a^{\rm T} $  denotes the tangent component of $a$ along $M^n$.
\begin{itemize}
\item[(i)]  The first or last equality holds if and only if $M^n$ is totally geodesic.
\item[(ii)] In the case of minimal CSC hypersurfaces, the second or third equality holds if and only if $M^n$ is an IE, non-totally geodesic, minimal CSC hypersurface.
Moreover, if $S>0$, i.e., $M^n$ is non-totally geodesic, then
\begin{equation}\label{ineq-MCSC-2n}
\frac{1}{2n}\leq \inf_{a\in \mathbb{S}^{n+1}}
\frac {\int_{M} \varphi^2_a}
{  {\rm Vol }\left( M^n \right)  },
\end{equation}
where the equality holds if and only if $M^n$ is $S^{1}(\sqrt{\frac{1}{n}})\times S^{n-1}(\sqrt{\frac{n-1}{n}})$.
\end{itemize}
\end{thm}

In fact, the left three inequalities of (\ref{equation  main Volume estimation of minimal hypersurface}) still hold without the minimal condition. Without the condition  of  constant scalar curvature in case $({\rm ii})$ of Theorem \ref{thm main Volume estimation of minimal hypersurface}, we also have the following Simons-type gap.
\begin{thm}\label{thm introduction Volume estimation of minimal hypersurface for nonconstant S}
Let $M^n$ be a closed minimal hypersurface immersed  in $\mathbb{S}^{n+1}$. Then
  \begin{itemize}
  \item [(i)]
\begin{equation*}
  \frac{1}
{2n}\int_{M}S
\leq
\sup_{p\in M^n}S(p)
\inf_{a\in \mathbb{S}^{n+1}}\int_{M} \varphi^2_a.
\end{equation*}
The equality holds if and only if $ S\equiv 0$ or $n$, and thus $M^n$ is either totally geodesic or the minimal Clifford torus $S^{1}(\sqrt{\frac{1}{n}})\times S^{n-1}(\sqrt{\frac{n-1}{n}})$.
\item [(ii)]
\begin{equation*}
\frac{n}{4n^2-3n+1} %{\cdot}
%\frac{\left( {\int_{M} }S\right) ^2}
%{\int_{M} S^2}
\left( {\int_{M} }S\right) ^2
\leq \int_{M} S^2
\inf_{a\in \mathbb{S}^{n+1}}\int_{M} \varphi^2_a.
\end{equation*}
The equality holds if and only if $M^n$ is totally geodesic.
\end{itemize}
\end{thm}

Let ${\rm Index}(M^n)$ denote the index of minimal hypersurfaces $M^n \subset \mathbb{S}^{n+1}$, the number of negative eigenvalues associated with the Jacobi (second variation) operator.
\begin{cor}\label{cor n in 2-6 Volume estimation of minimal hypersurface for nonconstant S}
Let $M^n$ $(2\leq n\leq 6)$ be a closed, non-totally geodesic, embedded minimal hypersurface in $\mathbb{S}^{n+1}$. Then there is a positive constant $C$ depending on  ${\rm Vol}(M^n)$ and  ${\rm Index}(M^n)$ such that
\begin{equation*}
\inf_{a\in \mathbb{S}^{n+1}}\int_{M} \varphi^2_a
\geq C{\rm Vol}(M^n).
\end{equation*}
In particular, if $n=2$,  the positive constant $C$  depends only  on  Euler characteristic $\chi$ of $M^2$.
\end{cor}

Lastly, we apply these inequalities to give some spherical Bernstein theorems.

\begin{defn}\label{def spherical zone}
 For any $a\in \mathbb{S}^{n+1}$ and $0<t<1$, we define the spherical zone as
 $$\mathbb{S}_{zone}^{n+1}(t)=\left\lbrace  x \in \mathbb{S}^{n+1}: |\left \langle x,a \right\rangle |<t \right\rbrace. $$
 % $$\mathbb{S}_{cap}^{n+1}(t)=\left\lbrace  \left( x_1,x_2,\cdots,x_{n+2} \right)  \in \mathbb{S}^{n+1}: |\left\langle x,a \right\rangle| >t \right\rbrace. $$
\end{defn}
It is well known that a closed minimal hypersurface lying in a closed hemisphere is totally geodesic (see Proposition \ref{thm 1spherical cap}).
Similarly, we have the following results for spherical zones.
\begin{thm}\label{thm  applitions of intergal formula}
Let $M^n$  be a closed  minimal hypersurface immersed in $\mathbb{S}^{n+1}$.
\begin{itemize}
\item[(i)]
If $M^n$ is CSC and the image of $M^n$ lying in $\mathbb{S}_{zone}^{n+1}(\sqrt{\frac{1}{2n}})$ (or in $\mathbb{S}_{zone}^{n+1}(\sqrt{\frac{1}{n+2}})$ when $M^n$ is IE), then it  is totally geodesic.
 %A closed immersed minimal  CSC hypersurface $M^n$ lying in $\mathbb{S}_{zone}^{n+1}(\sqrt{\frac{1}{2n}})$ (or in $\mathbb{S}_{zone}^{n+1}(\sqrt{\frac{1}{n+2}})$ if $M^n$ is IE) is totally geodesic.
\item[(ii)]   %A closed immersed minimal non-totally geodesic hypersurface
If $M^n$ is non-totally geodesic, then the image of $M^n$
 can not lie in $\mathbb{S}_{zone}^{n+1}(\sqrt{r})$, where  $r=\max\{r_1,r_2\}$ and
 $$r_1=\frac{\int_{M}S}
{2n {\rm Vol }\left( M^n
\right)
\sup_{p\in M^n}S(p)},\ \
r_2=\frac{n}{4n^2-3n+1}
 \frac{
 \left( {\int_{M} }S\right) ^2}
{ {\rm Vol }\left( M^n \right) \int_{M} S^2}.$$
\item[(iii)]
If $M^n$ $(2\leq n\leq 6)$ is embedded and non-totally geodesic, then  there is a positive constant ${\theta}$ depending on  ${\rm Vol}(M^n)$ and  ${\rm Index}(M^n)$ such that  the image of $M^n$
 can not lie in $\mathbb{S}_{zone}^{n+1}(\theta)$.
 In particular, if $n=2$,  ${\theta}$ depends only  on  Euler characteristic $\chi$ of $M^2$.
% A closed embedded  minimal   hypersurface $M^n$ lying in $\mathbb{S}_{zone}^{n+1}(\theta)$  is totally geodesic.
\end{itemize}
\end{thm}

\section{Takahashi-type theorem and characterizations for IE hypersurfaces}
Firstly, we recall the basic properties of the height functions $\varphi_a$ and $\psi_a$ defined in (\ref{height functions}), most of which were already in literature (cf. \cite{Nomizu and Smyth 1969}).

Let $x:M^{n}\rightarrow\mathbb{S}^{n+1}\subset\mathbb{R}^{n+2}$ be a closed hypersurface immersed in the unit sphere.
Let $\nabla,$  $\widetilde{\nabla}$ and $D$ be the Levi-Civita connections  on $M^n$, $\mathbb{S}^{n+1}$ and $\mathbb{R}^{n+2}$, respectively.
 Observe that the gradients of the height functions  are given by
$$\nabla\varphi_{a}(x)=a^{\rm T}, \quad \nabla\psi_{a}(x)=-A(a^{\rm T}),$$
where $a^{\rm T}\in \Gamma(TM)$ denotes the tangent component of $a$ along $M^n$, and $A$ is the shape operator with respect to the
 unit normal vector field $\nu$, i.e., $A(X)=-\widetilde{\nabla}_X\nu$.

Clearly, we can decompose the unit vector $a\in \mathbb{S}^{n+1}$ as
\begin{equation}\label{dec-a}
a= a^{\rm T}+\varphi_{a}(x)x+ \psi_{a}(x)\nu, \quad |a^{\rm T }|^2+\varphi_a^2+\psi_a^2=1.
\end{equation}

Since $D\varphi_{a}=a$, one deduces that the Hessian is
$$\mathrm{Hess}^\nabla\varphi_{a}(X,Y)=\mathrm{Hess}^D\varphi_{a}(X,Y)+B(X,Y)\varphi_{a}=0+\langle B(X,Y),a\rangle$$
for $X,Y\in \Gamma(TM)$ (the superscripts denote the connections). Here $B$ is the second fundamental form of $M^n$ as a submanifold in $\mathbb{R}^{n+2}$.
That is, $D_XY=\nabla_XY+B(X,Y)$.
Observe $$B(X,Y)=\langle B(X,Y),x\rangle x+\langle B(X,Y),\nu\rangle\nu=-\langle X,Y\rangle x+\langle A X,Y\rangle\nu.$$ Thus
\begin{equation*}
\mathrm{Hess}^\nabla\varphi_{a}(X,Y)=-\varphi_{a}(x)\langle X,Y\rangle+\psi_{a}(x)\langle A X,Y\rangle,
\end{equation*}
which, regarding $\mathrm{Hess}^\nabla \varphi_{a}$ as a $(1,1)$-tensor, can be rewritten as
\begin{equation*}
\mathrm{Hess}^\nabla\varphi_{a}=-\varphi_{a}(x){\rm Id}+\psi_{a}(x)A.
\end{equation*}
Hence
\begin{equation*}
\Delta\varphi_{a}(x)=-n\varphi_{a}(x)+nH\psi_{a}(x),
\end{equation*}
where $H:={\rm Tr}(A)/n$ is the mean curvature.

On the other hand,
\begin{eqnarray*}
\mathrm{Hess}^\nabla\psi_{a}(X)&:=&\nabla_X\nabla\psi_{a}=-\nabla_X\big(A(a^{\rm T})\big)=-(\nabla_XA)(a^{\rm T})-A(\nabla_Xa^{\rm T})\\
&=&-(\nabla_XA)(a^{\rm T})-A(\nabla_X\nabla\varphi_{a})=-(\nabla_XA)(a^{\rm T})-A(\mathrm{Hess}^\nabla\varphi_{a}(X))\\
&=&-(\nabla_{a^{\rm T}}A)(X)+\varphi_{a}(x)A(X)-\psi_{a}(x)A^2(X).
\end{eqnarray*}
Here the last equality follows from the Codazzi equation $(\nabla_{Y}A)(X)=(\nabla_{X}A)(Y)$.
Again we rewrite the Hessian as  a $(1,1)$-tensor
\begin{eqnarray*}
\mathrm{Hess}^\nabla\psi_{a}=-\nabla_{a^{\rm T}}A+\varphi_{a}(x) A-\psi_{a}(x) A^2.
\end{eqnarray*}

Therefore
\begin{equation*}
\Delta\psi_{a}=-{\rm Tr}(\nabla_{a^{\rm T}}A)+nH\varphi_{a}(x)-\|A\|^2\psi_{a}(x)=-n\langle\nabla H,{a}\rangle+nH\varphi_{a}(x)-\|A\|^2\psi_{a}(x).
\end{equation*}

In conclusion, we have shown
\begin{prop}\label{prop funda}
For a hypersurface $x: M^n\looparrowright    \mathbb{S}^{n+1} \subset  \mathbb{R}^{n+2}$ with the height functions $\varphi_a$ and $\psi_a$ defined in $(\ref{height functions})$, we have
$$\begin{array}{lll}
\nabla \varphi_a=a^{\rm T},&
\nabla \psi_a =-Aa^{\rm T},\\
\Delta \varphi_a=-n\varphi_a+nH\psi_a ,&
\Delta \psi_a =-n\left\langle \nabla H, a \right\rangle +nH\varphi_a -\|A\|^2\psi_a, \\
{\rm Hess}^\nabla\varphi_a=-\varphi_a {\rm Id}+\psi_a A, &
{\rm Hess}^\nabla\psi_a =-\nabla _{a^{\rm T}}A+\varphi_a A-\psi_a A^2.
\end{array}$$
\end{prop}

Now we are ready to prove the Takahashi-type Theorem.
\begin{proof}[\textbf{Proof of Theorem $\mathbf{\ref{thm the equivalent of minimal  CSC CMC hypersurfaces}}$}]
Case $({\rm i})$.
By Proposition \ref{prop funda}, one has
$$
\Delta \varphi_a=-n\varphi_a+nH\psi_a,
$$
which proves the necessity for $\lambda=n$ and $\mu=H$. Conversely, if
there exist some continuous function $\lambda$ and constant $\mu$ such that for all $a \in \mathbb{S}^{n+1}$,
$$
\Delta \varphi_a =-\lambda\varphi_a +n\mu\psi_a,
$$
then, combining the two equations above, we have
$$-nx+nH\nu=-\lambda x+n\mu\nu,$$
which shows $\lambda=n$ and $H=\mu$ are constant by the orthogonality of $x$ and $\nu$.

Case $({\rm ii})$.
By Proposition \ref{prop funda}, one has
$$
\Delta \psi_a =-n\left\langle \nabla H, a \right\rangle +nH\varphi_a-\|A\|^2\psi_a,
$$
which proves the necessity for $\lambda=\|A\|^2=S$, and $\mu=H$ is constant. Conversely, if
there exist some continuous functions $\lambda$ and $\mu$ such that for all $a \in \mathbb{S}^{n+1}$,
$$
\Delta \psi_a =-\lambda\psi_a +n\mu\varphi_a,
$$
then, combining the two equations above, we have
$$
-n\nabla H +nHx-\|A\|^2\nu=-\lambda\nu+
n\mu x,
$$
which shows $\nabla H=0$, $\lambda=\|A\|^2$, and $\mu=H$ is constant by the orthogonality of $\nabla H$, $x$ and $\nu$.
The same argument can prove the case $({\rm iii})$.
\end{proof}

Next we give the characterization of IE hypersurfaces in spheres.
\begin{proof}[\textbf{Proof of Theorem $\mathbf{\ref{thm-IE-hypers-charact}}$}]
Firstly, we recall Reilly's formula \cite{Reilly 1977}
\begin{equation}\label{equation Reilly's formular }
\int_{M}\Big(\left(\Delta f \right)^2-\|{\rm Hess}^\nabla f\|^2\Big)=
\int_{M}{\rm Ric}(\nabla f,\nabla f), \quad \textit{for any }f\in C^{\infty}(M).
\end{equation}
By  Proposition \ref{prop funda}, we have
\begin{equation*}
\left( \Delta \varphi_a\right) ^2=n^2\varphi_a^2+n^2H^2\psi_a^2-2n^2H\varphi_a\psi_a,
\end{equation*}
\begin{equation*}
\|{\rm Hess}^\nabla\varphi_a\|^2=n\varphi_a^2+\|A\|^2\psi_a^2-2nH\varphi_a\psi_a,
\end{equation*}
\begin{equation*}
\begin{aligned}
\frac{1}{2}\Delta  \varphi_a^2
&=\varphi_a\Delta  \varphi_a+
\left\langle \nabla \varphi_a,\nabla \varphi_a \right\rangle
=-n\varphi_a^2+nH\varphi_a\psi_a+|a^{\rm T }|^2\\
&=1-(n+1)\varphi_a^2-\psi_a^2+nH\varphi_a\psi_a,
\end{aligned}
\end{equation*}
where the last equality follows from (\ref{dec-a}) and it implies
\begin{equation}\label{IEleft}
\int_{M}\Big( 1 -(n+1)\varphi_a^2-\psi_a^2+nH\varphi_a\psi_a\Big) =0.
\end{equation}
Let $\rho=\frac{R}{n(n-1)}$ be the normalized scalar curvature. Then by the Gauss equation,
$$\rho-1=\frac{n^2H^2-\|A\|^2}{n(n-1)}.$$
Set $f(x)=\varphi_a(x)$ in (\ref{equation Reilly's formular }). By the preceding formulae and (\ref{dec-a}),  we calculate the IE integral (\ref{equation integral Einstein manifold}) as
\begin{eqnarray}\label{IEintegral}
&~&\int_{M}\left( {\rm Ric}-\frac{R}{n}\mathbf{g}\right)(a^{\rm T},a^{\rm T}) \\
&=&\int_{M}\left( \left(\Delta \varphi_a \right)^2-\|{\rm Hess}^\nabla\varphi_a\|^2-\frac{R}{n}\left\langle a^{\rm T},a^{\rm T}\right\rangle\right)  \nonumber\\
&=&\int_{M}\left( n^2\varphi_a^2+n^2H^2\psi_a^2-2n^2H\varphi_a\psi_a-\left(n\varphi_a^2+\|A\|^2\psi_a^2-2nH\varphi_a\psi_a\right) -\frac{R}{n}\|a^{\rm T}\|^2\right) \nonumber\\
&=&\int_{M}\Big(\left(n^2-1+(n-1)(\rho-1)\right) \varphi_a^2+\left( n-1+(n^2-1)(\rho-1)\right)\psi_a^2\Big) - \nonumber\\
&~&\ \ \ \ \int_{M}\left( 2n(n-1)H\varphi_a\psi_a+(n-1)\rho\right) \nonumber \\
&=&\left( n-1\right)\Big(\int_{M}\left( 1-(n+1)\varphi_a^2-\psi_a^2\right)-\int_{M}\left( \rho-1\right)\left( 1-\varphi_a^2-(n+1)\psi_a^2\right)\Big). \nonumber
\end{eqnarray}
This shows the equivalence between the IE equation (\ref{equation integral Einstein manifold}) and (\ref{equation integral Einstein  hypersurface}).

For the proof of case $(A)$, i.e., $H=0$, it follows from (\ref{IEleft}) that the left hand of (\ref{equation integral Einstein  hypersurface}) vanishes, namely,
\begin{equation}\label{IEleft-minimal}
\int_{M}\Big( 1 -(n+1)\varphi_a^2-\psi_a^2\Big) =0.
\end{equation}
Therefore, as $\rho-1=-S/(n(n-1))$, (\ref{equation integral Einstein  hypersurface}) is equivalent to $$
 \int_{M}S
\left(  1-\varphi_a^2-(n+1)\psi_a^2\right) =0.
$$

For the proof of case $(B)$, i.e., $H=0$ and $S=\|A\|^2\equiv \mathrm{Const}$, the formula (\ref{IEintegral-MCSC}) follows easily from (\ref{IEleft-minimal}) and (\ref{IEintegral}). Then the subcases $({\rm i}), ({\rm ii})$ and $({\rm iii})$ of case $(B)$ follow directly from (\ref{IEintegral-MCSC}) and (\ref{IEleft-minimal}).
%we observe that by case $(A)$, (\ref{equation integral Einstein  hypersurface}) is equivalent to
%$$
% \int_{M}\Big(  1-\varphi_a^2-(n+1)\psi_a^2\Big) =0.
%$$
%Then combining with (\ref{IEleft-minimal}) shows the equivalence between (\ref{equation integral Einstein  hypersurface}) and the subcase $({\rm iii})$:
%$$
%\int_{M}\varphi_a^2=\int_{M}\psi_a^2,
%$$
%which proves the subcases $({\rm i})$ and $({\rm ii})$ of case $(B)$ by combining with (\ref{IEleft-minimal}) again.

The last subcase $({\rm iv})$ of case $(B)$ is intriguing but useful in deriving Corollary \ref{cor-f3-isop}. In the following
we give a simple proof by the self-adjoint operator of Cheng-Yau \cite{Cheng Yau Math Ann 1977}. We briefly recall the Cheng-Yau operator as follows.
For a $C^2$-function $f$ on $M^n$, the gradient $\nabla f=\sum_{i}f_i e_i$ and the Hessian ${\rm Hess}^\nabla f=\sum_{i,j} f_{ij}\omega_i\otimes \omega_j$ of $f$ under a local orthonormal frame $\{e_i\}_{i=1}^n$ can be computed by
$$
df=\sum_{i}f_i\omega_i,\ \
\sum_{j}f_{ij}\omega_j=df_i+\sum_{j}f_j\omega_{ji},
$$
where $\{\omega_i\}_{i=1}^n$ is the coframe and $\{\omega_{ji}\}$ are the connection forms.
The covariant derivative $\phi_{ijk}$ of a $2$-tensor $\phi_{ij}$ is defined by
$$
\sum_{k}\phi_{ijk}\omega_k=d\phi_{ij}+\sum_{k}\phi_{kj}\omega_{ki}+\sum_{k}\phi_{ik}\omega_{kj}.
$$
Let $\phi=\sum_{i,j} \phi_{ij}\omega_i\otimes \omega_j$ be a symmetric tensor on $M^n$. The Cheng-Yau operator associated to $\phi$ is defined by
$$
\square f=\sum_{i,j}\phi_{ij}f_{ij}=\langle \phi, {\rm Hess}^\nabla f \rangle.
$$
Then if $M^n$ is a closed manifold, by Stokes' theorem, for any $C^2$-function $u$ on $M^n$,
$$
 \begin{aligned}
\int_{M}\left( \square f\right) u
&=\int_{M}\sum_{i,j}\phi_{ij}f_{ij}u=
-\int_{M}\sum_{i,j}
\left( \phi_{ij}u\right) _jf_{i}\\
&=-\int_{M}\sum_{i,j}
\left( \phi_{ijj}u+\phi_{ij}u_{j}\right) f_{i}\\
&=\int_{M}\sum_{i,j}\phi_{ij}u_{ij}f
+\int_{M}\sum_{i,j}
\phi_{ijj}\left( fu_{i}-uf_{i}\right) \\
&=\int_{M}f \left( \square u\right)+\int_{M}\sum_{i,j}
\phi_{ijj}\left( fu_{i}-uf_{i}\right).
 \end{aligned}
$$
Thus,  the operator $\square$ is self-adjoint if and only if
$$
\sum_{j}\phi_{ijj}=0,
$$
for all $i$ (\cite{Cheng Yau Math Ann 1977}). Cheng-Yau provided two symmetric tensors satisfying the preceding condition, namely,
$$\phi_{ij}=\frac{R}{2}\delta_{ij}-{\rm Ric}_{ij}, \quad \textit{or} \quad  \phi_{ij}=({\rm Tr}\Psi)\delta_{ij}-\Psi_{ij},$$
where $\Psi$ is a symmetric Codazzi tensor. Now in our case $(B)$, $R$ is constant, ${\rm Tr}A=nH=0$ and the shape operator $A$ is Codazzi. Then both of the following two tensors
$$\phi_{ij}=\frac{R}{n}\delta_{ij}-{\rm Ric}_{ij}, \quad \textit{or} \quad \phi=A,$$
give rise to a self-adjoint Cheng-Yau operator, either of which can help to prove the subcase $({\rm iv})$ of case $(B)$. We proceed with the proof by the second for example.

By  Proposition \ref{prop funda}, we have
$$\begin{aligned}
&\int_M \psi_a (\square \varphi_a)=\int_M \psi_a  \langle A, -\varphi_a {\rm Id}+\psi_a A\rangle=\int_M \psi_a^2 S,\\
&\int_M \varphi_a (\square \psi_a)=\int_M \varphi_a \langle A, -\nabla _{a^{\rm T}}A+\varphi_a A-\psi_a A^2\rangle=\int_M (\varphi_a^2 S-\varphi_a\psi_a {\rm Tr} A^3),
 \end{aligned}$$
 which by the self-duality of the Cheng-Yau operator implies
 $$\int_M (\varphi_a^2 -\psi_a^2 )S=\int_M\varphi_a\psi_a {\rm Tr} A^3.$$
 The proof is completed by taking use of the subcase $({\rm iii})$.
\end{proof}

\begin{proof}[\textbf{Proof of Corollary $\mathbf{\ref{cor-f3-isop}}$}]
For  minimal CSC hypersurfaces in $\mathbb{S}^{n+1}$  with $S>n$, we know from Proposition \ref{prop funda} that $\varphi_a$ and $\psi_a$ are eigenfunctions of the Laplacian to the different eigenvalues $n$ and $S$ respectively. Therefore, they are orthogonal and thus the condition in the subcase $({\rm iv})$ of case $(B)$ of Theorem \ref{thm-IE-hypers-charact} is satisfied, namely, $$\int_M\varphi_a\psi_a {\rm Tr} A^3=({\rm Tr} A^3)\int_M\varphi_a\psi_a=0,$$
if further $M^n$ has constant third mean curvature (and thus constant ${\rm Tr} A^3$).

In particular, minimal isoparametric hypersurfaces with $g\geq3$ principal curvatures have constant mean curvatures of each order and have constant $S=(g-1)n>n$, thus they are IE hypersurfaces in unit spheres.

The Einstein isoparametric hypersurfaces (with $g=2$) $S^{k}(\sqrt{\frac{k-1}{n-2}})\times S^{n-k}(\sqrt{\frac{n-k-1}{n-2}})$ $\left( 2 \leq k \leq n-2\right)$ are automatically IE hypersurfaces in $\mathbb{S}^{n+1}$. However, the only left Clifford torus $M^n:=S^{1}(r_1)\times S^{n-1}(r_2)\subset\mathbb{S}^{n+1}$ $(0<r_1<1, r_1^2+r_2^2=1)$ is not an IE hypersurface in $\mathbb{S}^{n+1}$. The proof is a long but straightforward calculation of the integrals in both sides of (\ref{equation integral Einstein  hypersurface}). Here we leave the details to the reader and only give the final result of the calculation of the two sides of (\ref{equation integral Einstein  hypersurface}) as follows
$$\begin{aligned}
& {\rm LHS}  = \Big(1-(n+1)(r_1^2\frac{|a_1|^2}{2}+r_2^2\frac{|a_2|^2}{n})-(r_2^2\frac{|a_1|^2}{2}+r_1^2\frac{|a_2|^2}{n})\Big)V,\\
& {\rm RHS}=\frac{nr_1^2-2}{r_2^2n}\Big(1-(r_1^2\frac{|a_1|^2}{2}+r_2^2\frac{|a_2|^2}{n})-(n+1)(r_2^2\frac{|a_1|^2}{2}+r_1^2\frac{|a_2|^2}{n})\Big)V,
 \end{aligned}$$
 where $V={\rm Vol}(M^n)=r_1r_2^{n-1}{\rm Vol}(\mathbb{S}^1){\rm Vol}(\mathbb{S}^{n-1})$ and $a=(a_1,a_2)\in\mathbb{R}^2\oplus\mathbb{R}^n$. Direct calculations can show that the equation (\ref{equation integral Einstein  hypersurface}), i.e., $ {\rm LHS}  =  {\rm RHS}$ does not hold for all $a\in\mathbb{S}^{n+1}$. This completes the proof by Theorem \ref{thm-IE-hypers-charact}.
\end{proof}

\section{Integral inequalities with equalities by IE hypersurfaces}
In this section, based on the previous arguments, we estimate uniformly the $L^2$ squared norm of the position height function $\varphi_a$ of (\ref{height functions}) by further considering the height functions $\varphi_{a_j}$ with respect to an orthonormal frame  $\{a_j\}_{j=1}^{n+2}$ of $\mathbb{R}^{n+2}$.

\begin{proof}[\textbf{Proof of Theorem $\mathbf{\ref{thm main Volume estimation of minimal hypersurface}}$}]
The first inequality of (\ref{equation  main Volume estimation of minimal hypersurface}) is obvious and attains equality only at the totally geodesic hyperspheres $\{x\in\mathbb{S}^{n+1} : \varphi_a(x)=0\}$.

For the second and third inequality of (\ref{equation  main Volume estimation of minimal hypersurface}), we consider the height functions $\varphi_{a_j}$ with respect to an orthonormal frame $\{a_j\}_{j=1}^{n+2}$ of $\mathbb{R}^{n+2}$. It is easily seen that
$$\sum_{j=1}^{n+2}\varphi_{a_j}^2=1, \quad \sum_{j=1}^{n+2}\int_M\varphi_{a_j}^2={\rm Vol}(M^n),$$
which directly shows
$$(n+2)\inf_{a\in \mathbb{S}^{n+1}}
\int_{M} \varphi^2_a
\leq {\rm Vol }\left( M^n \right)\leq(n+2)\sup_{a\in \mathbb{S}^{n+1}}
\int_{M} \varphi^2_a.$$
In the case of minimal CSC hypersurfaces, the equalities above hold if and only if $(n+2)\int_{M} \varphi^2_a={\rm Vol}(M^n)$ for all $a\in\mathbb{S}^{n+1}$, i.e., $M^n$ is an IE (non-totally geodesic) minimal CSC hypersurface by the case $(B)$ of Theorem \ref{thm-IE-hypers-charact}.

The last inequality of (\ref{equation  main Volume estimation of minimal hypersurface}) follows easily from (\ref{IEleft-minimal}) (when $M^n$ is minimal), namely,
$$(n+1)\int_{M}\varphi_a^2=\int_{M}\Big( 1 -\psi_a^2\Big)\leq {\rm Vol}(M^n),$$
which attains equality if and only if $\psi_{a_0}\equiv 0$ for some $a_0\in\mathbb{S}^{n+1}$, i.e., the Gauss image of $M^n$ lies in an equator of $\mathbb{S}^{n+1}$, and in this case $M^n$ is embedded as an equator in $\mathbb{S}^{n+1}$ by a theorem of Nomizu and Smyth \cite{Nomizu and Smyth 1969}.

Now we come to prove the inequality (\ref{ineq-MCSC-2n}). When $M^n$ is minimal, by Proposition \ref{prop funda}, we have
$$\frac{1}{2}\Delta \psi_a^2=-S\psi_a^2+|Aa^{\rm T}|^2,$$
and thus
\begin{equation}\label{deltapsi2}
\int_M S\psi_a^2=\int_M |Aa^{\rm T}|^2.
\end{equation}
Let $\{\lambda_i\}_{i=1}^n$  be the eigenvalues of $A$ with $\lambda_1^2\geq\lambda_2^2\geq\cdots\geq\lambda_n^2$. Then we have $$\sum_{i=1}^{n}\lambda_i=0,\quad \sum_{i=1}^{n}\lambda_i^2=\|A\|^2=S.$$
Thus
$$
\begin{aligned}
0=\Big(\sum_{i=1}^{n}\lambda_i\Big)^2
&=\lambda_1^2+
2\lambda_1\sum_{i=2}^{n}\lambda_i+
\Big(\sum_{i=2}^{n}\lambda_i\Big)^2=-\lambda_1^2+\Big(\sum_{i=2}^{n}\lambda_i\Big)^2\\
&\leq-\lambda_1^2+(n-1)\sum_{i=2}^{n}\lambda_i^2=(n-1)S-n\lambda_1^2.
\end{aligned}
$$
Hence
\begin{equation}\label{eigen1est}
\lambda_1^2\leq\frac{n-1}{n}S,
\end{equation}
where the equality holds if and only if $ \lambda_1=(1-n)\lambda_2$ and $\lambda_2=\lambda_3=\dots=\lambda_n$.

It follows from (\ref{deltapsi2}) and (\ref{eigen1est}) that
\begin{equation}\label{equation intpsi1}
\int_M S\psi_a^2=\int_M |Aa^{\rm T}|^2 \leq
\int_{M}\lambda_1^2|a^{\rm T}|^2\leq
\frac{n-1}{n}\int_{M}S|a^{\rm T}|^2.
\end{equation}
On the other hand, by (\ref{IEleft-minimal}) and (\ref{dec-a}) we have
\begin{equation}\label{equa-aTnphi}
\int_{M}|a^{\rm T}|^2=n\int_{M}\varphi_a^2.
\end{equation}
Combining this with (\ref{equation intpsi1}) and (\ref{dec-a}), we obtain
 \begin{equation}\label{equintS}
\int_M S=\int_M S(\varphi_a^2+\psi_a^2+|a^{\rm T}|^2)\leq 2n \sup_{p\in M^n}S(p)\int_{M}\varphi_a^2,
\end{equation}
which proves the inequality (\ref{ineq-MCSC-2n}) if $S>0$ is constant. The equality of (\ref{equintS}) holds for some $a$ only if $S\equiv \mathrm{Const}$ and equalities hold in (\ref{eigen1est}, \ref{equation intpsi1}), which thus implies that $M^n$ is either totally geodesic or a minimal hypersurface with two constant distinct principal curvatures $\lambda_1$ and $\lambda_2$ of multiplicities $1$ and $n-1$ respectively. Hence, when $S>0$, $M^n$ is the minimal Clifford torus $M_{1,n-1}:=S^{1}(\sqrt{\frac{1}{n}})\times S^{n-1}(\sqrt{\frac{n-1}{n}})$. It is left to verify that there exists some $a$ such that the equality of (\ref{equintS}) holds on $M_{1,n-1}$, i.e.,
$${\rm Vol }\left( M_{1,n-1} \right) =2n\int_{M_{1,n-1}}\varphi_a^2.$$
Let $a=(a_1,a_2)$ with $\|a\|=1$, where $a_1\in \mathbb{R}^{2}$ and $a_2=0\in \mathbb{R}^{n}$. Write $x=(x_1,x_2)\in S^{1}(\sqrt{\frac{1}{n}})\times S^{n-1}(\sqrt{\frac{n-1}{n}})$. Then the verification can be done directly by
\begin{equation*}
\begin{aligned}
2n\int_{M_{1,n-1}}\varphi_a^2
&=2n\int_{M_{1,n-1}}\left\langle a,x    \right\rangle^2\\
&=2n\int_{S^{1}(\sqrt{\frac{1}{n}})\times S^{n-1}(\sqrt{\frac{n-1}{n}})}\left\langle a_1,x_1    \right\rangle^2\\
&=2n {\rm Vol }\left( S^{n-1}(\sqrt{\frac{n-1}{n}}) \right)
\int_{S^{1}(\sqrt{\frac{1}{n}})}\left\langle a_1,x_1    \right\rangle^2\\
&={\rm Vol }\left( S^{n-1}(\sqrt{\frac{n-1}{n}}) \right)
{\rm Vol }\left( S^{1}(\sqrt{\frac{1}{n}}) \right) \\
&= {\rm Vol }\left( M_{1,n-1} \right).
\end{aligned}
\end{equation*}
%The proof is now complete.
\end{proof}

\begin{proof}[\textbf{Proof of Theorem $\mathbf{\ref{thm introduction Volume estimation of minimal hypersurface for nonconstant S}}$}]
The first case has been proven previously in (\ref{equintS}).

For the second case, by (\ref{dec-a}, \ref{equation intpsi1}, \ref{equa-aTnphi}) and the Cauchy inequality, we have
$$
\begin{aligned}
\int_{M}S
&\leq
\int_{M}S\left(\varphi_a^2  +\frac{2n-1}{n} |a^{\rm T }|^2\right) \\
&\leq\int_{M}\Big(S^2\varphi^2_a+S^2|a^{\rm T }|^2\Big)^{\frac{1}{2}}
\left(\varphi^2_a+
\Big( \frac{2n-1}{n} \Big) ^2|a^{\rm T }|^2\right)^{\frac{1}{2}}\\
&\leq \int_{M} S\left(\varphi^2_a+
\Big( \frac{2n-1}{n}\Big) ^2|a^{\rm T }|^2 \right)^{\frac{1}{2}}\\
&\leq\left(\int_{M}S^2\right)^{\frac{1}{2}}
\left(\int_{M} \frac{4n^2-3n+1}{n}\varphi^2_a\right)^{\frac{1}{2}}.
\end{aligned}
$$
The equality holds if and only if $S\equiv0$, or $|a^{\rm T }|=\psi_a\equiv0$ which implies also $M^n$ is totally geodesic.
\end{proof}

To prove Corollary \ref{cor n in 2-6 Volume estimation of minimal hypersurface for nonconstant S}, we need the following lemmas.
\begin{lem}[Choi-Schoen \cite{H I Choi and R Schoen 1985}]
\label{lemma Choi-Schoen-Smax}
 Assume $N^{3}$ is a closed Riemannian manifold with positive Ricci curvature. If  $M^2$ is a compact embedded minimal surface of $N^{3}$, then  there exists a constant $C_E$ depending only on $N^{3}$  and  Euler characteristic $\chi$ of $M^2$ such that
$$
\sup_{p\in M^2}S(p)\leq C_E.
$$
%where $\|A\|$ is the length of the second fundamental form.
\end{lem}

\begin{lem}[Sharp \cite{Sharp Ben 2017}]
\label{lemma  Sharp-Smax}
 Assume $N^{n+1}$ $(2 \leq n \leq 6)$ is a closed Riemannian manifold with positive Ricci curvature. If  $M^n$ is a compact embedded minimal hypersurface of $N^{n+1}$, then  there exists a constant $C_1$ depending only on $N^{n+1}$, ${\rm Vol}(M^n)$ and  ${\rm Index}(M^n)$  such that
$$
\sup_{p\in M^n}S(p)\leq C_1.
$$
%In particular, if $n=2$, then the  constant $C_1$  depends only  on  Euler characteristic $\chi$ of $M^2$.
%where $\|A\|$ is the length of the second fundamental form.
\end{lem}

%Corollary 2.6. Let $N^{n+1}$ be a closed Riemannian manifold with Ric $_N>0$ and $2 \leq n \leq 6$. The class $\mathcal{M}(\Lambda, I)$ is compact in the $C^k$ topology for all $k \geq 2$.

%Remark 2.7. Notice that by an easy argument we have the existence of some $C=C(\Lambda, I, N)$ such that for any $M \in \mathcal{M}(\Lambda, I)$ where $\operatorname{Ric}_N>0$ (see also $[4$, Theorem 2])
%$$
%\sup _M|A| \leq C
%$$

 \begin{lem}[Ge-Li \cite{Ge Li 2020 A lower bound second fundamental form}]\label{Lemma Ge-Li Predomo conjecture}
 Let $M^n$ be a closed embedded, non-totally geodesic, minimal hypersurface in $\mathbb{S}^{n+1}$.
 Then there is  a positive constant
  $C_2$, depending only on $n$,
  such that
 $$\int_{M}S \geq C_2{\rm Vol}(M^n).$$
 \end{lem}

\begin{proof}[\textbf{Proof of Corollary $\mathbf{\ref{cor n in 2-6 Volume estimation of minimal hypersurface for nonconstant S}}$}]
%Without loss of generality, we suppose  $M^n$ is non-totally geodesic.
By Lemmas \ref{lemma  Sharp-Smax} - \ref{Lemma Ge-Li Predomo conjecture} and Theorem \ref{thm introduction Volume estimation of minimal hypersurface for nonconstant S} (i), we have
\begin{equation*}
\inf_{a\in \mathbb{S}^{n+1}}\int_{M} \varphi^2_a
\geq
  \frac{\int_{M}S}
{2n\sup_{p\in M^n}S(p)}\geq \frac{C_2}{2nC_1}{\rm Vol}(M^n).
\end{equation*}
In particular,
by the Gauss equation and the Gauss-Bonnet theorem, for genus $g$ minimal surface $M^2\subset\mathbb{S}^3$, we have $\chi=2-2g$ and
 $$\int_{M}S=8\pi\left( g-1\right) +2{\rm Vol}(M^2).$$
 If $g=0$, Calabi \cite{Calabi 1967} proved that \emph{if ${S}^{2}$ is minimally immersed in $\mathbb{S}^{3}$, then ${S}^{2}$ is an equator (i.e,  totally geodesic).}
 If $g=1$, Brendle \cite{Brendle S 2013} verified Lawson's Conjecture, i.e.,
  \emph{the only embedded minimal torus in $\mathbb{S}^3$ is the Clifford torus.}
 For $g\geq2$, by Lemma  \ref{lemma Choi-Schoen-Smax} and Theorem \ref{thm introduction Volume estimation of minimal hypersurface for nonconstant S} (i), one has
$$
\inf_{a\in \mathbb{S}^{3}}\int_{M} \varphi^2_a
\geq
 \frac{4\pi\left( g-1\right) +{\rm Vol}(M^2)}{2C_E}\geq  \frac{4\pi+{\rm Vol}(M^2)}{2C_E}
 .$$
In conclusion, for surface case, there is a positive constant $C>0$ depending only on the Euler characteristic $\chi$ of $M^2$ such that
  \begin{equation*}
\inf_{a\in \mathbb{S}^{n+1}}\int_{M} \varphi^2_a
\geq C{\rm Vol}(M^2).
\end{equation*}
\end{proof}

To conclude this section, we give the uniform bounds of the $L^2$ squared norm of $\varphi_a$ on minimal isoparametric hypersurfaces, which give another proof of the result of Corollary \ref{cor-f3-isop}: minimal isoparametric hypersurfaces with $g\geq3$ are IE hypersurfaces.
\begin{thm}\label{thm Volume estimation of minimal isoparametric hypersurfaces}
Let $M^n$ be a minimal isoparametric hypersurface   with $g\geq2$ distinct principal curvatures in  $\mathbb{S}^{n+1}$.
\begin{itemize}
\item[(i)] For $g=2$, i.e.,  $M^n=S^{k}(\sqrt{\frac{k}{n}})\times S^{n-k}(\sqrt{\frac{n-k}{n}})$, $(1\leq k\leq [\frac{n}{2}])$,  we have
$$
\inf_{a\in \mathbb{S}^{n+1}}
\frac {\int_{M} \varphi^2_a}
{  {\rm Vol }\left( M^n \right)  }
= \frac{k}{n(k+1)}, \quad
\sup_{a\in \mathbb{S}^{n+1}}
\frac {\int_{M} \varphi^2_a}
{  {\rm Vol }\left( M^n \right)  }
= \frac{n-k}{n(n-k+1)};
$$
\item[(ii)] For $g\geq3$, we have
$$
\int_{M} \varphi^2_a=\int_{M} \psi^2_a=\frac{1}{n+2}{\rm Vol}(M^n),
$$
 for  all  $a\in \mathbb{S}^{n+1}$, and thus $M^n$ is an IE minimal CSC hypersurface.
\end{itemize}
\end{thm}
\begin{rem}
In fact, for $g=4$, on each isoparametric hypersurface (not only minimal) we have $\int_{M} \varphi^2_a=\int_{M} \psi^2_a$.
\end{rem}

\begin{proof}
Case $({\rm i})$.
  Let $a=(a_1,a_2)\in \mathbb{S}^{n+1}$ and $x=(x_1,x_2)\in S^{k}(\sqrt{\frac{k}{n}})\times S^{n-k}(\sqrt{\frac{n-k}{n}})$, where $a_1\in \mathbb{R}^{k+1}$ and $a_2\in \mathbb{R}^{n+1-k}$. Then we have
$$
\begin{aligned}
&\ \ \ \ \int_{M}\varphi^2_a=\int_{M}\left\langle a,x    \right\rangle^2=
\int_{M}\left\langle a_1,x_1    \right\rangle^2+\left\langle a_2,x_2    \right\rangle^2\\
&= {\rm Vol }\left( S^{n-k}(\sqrt{\frac{n-k}{n}}) \right)
\int_{S^{k}(\sqrt{\frac{k}{n}})}\left\langle a_1,x_1    \right\rangle^2+
{\rm Vol }\left( S^{k}(\sqrt{\frac{k}{n}}) \right)
\int_{S^{n-k}(\sqrt{\frac{n-k}{n}})}\left\langle a_2,x_2    \right\rangle^2\\
&=\left(\frac{k}{n(k+1)}\|a_1\|^2+ \frac{n-k}{n(n-k+1)}\|a_2\|^2\right)  {\rm Vol }\left(M^n\right),
\end{aligned}
$$
which gives the uniform bounds immediately.

Case $({\rm ii})$. %For $g=3$, we observe that $\nu: M^n\to M_+$ is the submersion with fibre $\mathbb{S}^m$, where $M_+\cong \mathbb{FP}^2$ is one of the two focal submanifolds for $\mathbb{F}=\mathbb{R}, \mathbb{C}, \mathbb{H}, \mathbb{O}$ with respect to the common multiplicity $m=1,2,4$, or $8$ of the three principal curvatures. Thus
%$$
%\int_{M} \varphi^2_a=\int_{M_+}\int_{\mathbb{S}^m}\varphi^2_a={\rm Vol}\left( \mathbb{S}^m\right) \int_{M_+}\varphi^2_a,$$
%$$
%\int_{M} \psi^2_a=\int_{M_+}\int_{\mathbb{S}^m}\psi^2_a={\rm Vol}\left( \mathbb{S}^m\right) \int_{M_+}\psi^2_a.
%$$
%Since $\int_{M_+}\varphi^2_a=\int_{M_+}\psi^2_a$, we obtain $\int_{M}\varphi^2_a=\int_{M}\psi^2_a$.
For  $g\geq4$, we observe that $\nu: M^n\hookrightarrow M^n$ is a diffeomorphism (if $M^n$ is minimal when $g=6$) and $|\det d \nu|=1$, thus $\int_{M}\varphi^2_a=\int_{M}\psi^2_a$ by diffeomorphism invariance of integration. This shows that $M^n$ is IE if $M^n$ is minimal, by the subcase $({\rm iii})$ of case $(B)$ in Theorem \ref{thm-IE-hypers-charact} (other than $({\rm iv})$ as in the proof of Corollary \ref{cor-f3-isop}), or by the case $({\rm ii})$ of Theorem \ref{thm main Volume estimation of minimal hypersurface} since now $\int_{M}\varphi^2_a\equiv{\rm Vol}(M^n)/(n+2)$ by (\ref{IEleft-minimal}).

In general, for $g\geq3$ minimal isoparametric hypersurfaces in $\mathbb{S}^{n+1}$, we give one more proof by using the isoparametric theory and the integral inequality (\ref{equation  main Volume estimation of minimal hypersurface}) of Theorem \ref{thm main Volume estimation of minimal hypersurface}. Recall (cf. \cite{CR15}, \cite{Ge19}) that now $M^n=M_{\theta_0}=f^{-1}(c_0)$ is the minimal level hypersurface of the Cartan-M\"{u}nzner isoparametric function $f$ on $\mathbb{S}^{n+1}$, where $c_0=\frac{m_--m_+}{m_-+m_+}=\cos(g\theta_0)$, $0<\theta_0<\frac{\pi}{g}$, and $f(p)=\cos(g\theta(p))$ with $\theta(p)$ the distance of $p\in\mathbb{S}^{n+1}$ to $M_+$ (one of the two focal submanifolds $M_{\pm}:=f^{-1}(\pm1)$ with codimensions $m_{\pm}+1$). Moreover, the parallel level sets $M_\theta:=f^{-1}(\cos(g\theta))$, $\theta\in[0,\frac{\pi}{g}]$ (with $M_0=M_+, M_{\frac{\pi}{g}}=M_-$), constitute a singular Riemannian foliation of $\mathbb{S}^{n+1}$. Hence we have
\begin{equation}\label{equation Mtheta 0gpi}
\int_{0}^{\frac{\pi}{g}}
\int_{ M_{\theta}}\varphi_a^2=
\int_{x\in\mathbb{S}^{n+1}}\varphi_a^2(x)=\frac{1}{n+2}{\rm Vol}\left(\mathbb{S}^{n+1} \right).
\end{equation}
For $\theta\in(0,\frac{\pi}{g})$, the following spherical parallel translation is a diffeomorphism:
$$\begin{aligned}
\phi_{\theta}: M_{\theta_0}&\longrightarrow M_{\theta}\\
x&\longmapsto \cos(\theta_0-\theta)x+
 \sin(\theta_0-\theta)\nu.
\end{aligned}$$
It follows that (cf. \cite{CR15})
$$\phi_{\theta}^*(d{\rm Vol}_{M_{\theta}})=h(\theta)d{\rm Vol}_{M_{\theta_0}},$$
where $h(\theta) =\prod_{i=1}^n\left(\cos(\theta_0-\theta)-\sin(\theta_0-\theta)\lambda_i \right)$, and $\{\lambda_i\}_{i=1}^n$ are the constant principal curvatures of $M_{\theta_0}$ with the $g$ distinct values $\{\cot(\theta_0+\frac{(j-1)\pi}{g})\}_{j=1}^g$ of multiplicities $m_+$ and $m_-$ alternately.
%$$\cot(\theta_0)=\lambda_1=\cdots=\lambda_{m_+}>\lambda_{m_++1}=\cdots=\lambda_{m_++m_-}=\cot(\theta_0+\frac{\pi}{g})>\cdots \cot(\theta_0+\frac{(g-1)\pi}{g}).$$

Since $g\geq3$, $S=n(g-1)$ and by Proposition \ref{prop funda} on $M_{\theta_0}$,
\begin{equation*}
\Delta\varphi_a=-n\varphi_a,\quad
\Delta\psi_a=-n(g-1)\psi_{a},
\end{equation*}
we have
$\int_{M_{\theta_0}}\psi_a
\varphi_a=0$. Therefore
\begin{equation}\label{equation Mtheta htheta}
\begin{aligned}
\int_{M_{\theta}}\varphi_a^2
&=\int_{M_{\theta_0}}\left\langle \cos(\theta_0-\theta)x+
 \sin(\theta_0-\theta)\nu,a \right\rangle^2 |h(\theta)|\\
 &=\int_{M_{\theta_0}}\left(\cos^2(\theta_0-\theta)\varphi_a^2+\sin^2(\theta_0-\theta)\psi_a^2 +\sin2(\theta_0-\theta)\varphi_a\psi_a\right) |h(\theta)|\\
  &=\int_{M_{\theta_0}}\left(\cos^2(\theta_0-\theta)\varphi_a^2+\sin^2(\theta_0-\theta)\psi_a^2 \right) |h(\theta)|\\
  &=\int_{M_{\theta_0}}\left(\Big(1-(n+2)\sin^2(\theta_0-\theta)\Big)\varphi_a^2+\sin^2(\theta_0-\theta) \right) |h(\theta)|,
 \end{aligned}
\end{equation}
where the last equality follows from (\ref{IEleft-minimal}). By (\ref{equation Mtheta 0gpi}) and (\ref{equation Mtheta htheta}), we have
\begin{equation}\label{eq-L2norm-phi}
\Big(\beta-(n+2)\alpha\Big)\int_{M_{\theta_0}}\varphi_a^2+\alpha {\rm Vol}(M_{\theta_0})=\frac{1}{n+2}{\rm Vol}\left(\mathbb{S}^{n+1} \right),
\end{equation}
where $\alpha=\int_0^{\frac{\pi}{g}}\sin^2(\theta_0-\theta)|h(\theta)|d\theta$, and $\beta=\int_0^{\frac{\pi}{g}}|h(\theta)|d\theta$.  Analogous to (\ref{equation Mtheta 0gpi}), we have
$$\beta{\rm Vol}(M_{\theta_0})={\rm Vol}\left(\mathbb{S}^{n+1}\right).$$
It follows from (\ref{eq-L2norm-phi}) that either $\beta-(n+2)\alpha=0$ or $\int_{M_{\theta_0}}\varphi_a^2={\rm Vol}(M_{\theta_0})/(n+2)$.
So we are left with proving that the former equality is impossible for $g\geq3$.
 Here we take the case $g=3$ and $m_{\pm}=1$ for example and leave the other cases to the reader.
Now $\theta_0=\frac{\pi}{6}$, $\lambda_1=\cot\frac{\pi}{6}=\sqrt{3}$, $\lambda_2=\cot\frac{\pi}{2}=0$, $\lambda_3=\cot\frac{5\pi}{6}=-\sqrt{3}$, and thus
$$\beta-(n+2)\alpha=\int_{-\frac{\pi}{6}}^{\frac{\pi}{6}}(1-5\sin^2t)(\cos^2t-3\sin^2t)\cos tdt=\frac{1}{2}.$$
\end{proof}

\section {Applications to spherical Bernstein theorems}

In this section, we apply the integral inequalities of Theorems \ref{thm main Volume estimation of minimal hypersurface} and \ref{thm introduction Volume estimation of minimal hypersurface for nonconstant S} to  prove Theorem \ref{thm  applitions of intergal formula} for spherical zone domains. Firstly we recall the following classical spherical Bernstein theorem for hemispheres.

\begin{prop}\label{thm 1spherical cap} %\cite{Nomizu and Smyth 1969}
%A closed minimal hypersurface lying in a closed hemisphere is totally geodesic.
Let $ M^n$ be a closed minimal hypersurface lying in a closed hemisphere $\mathbb{S}_+^{n+1}:=\{x\in\mathbb{S}^{n+1}: \varphi_a(x)\geq0\}$. Then $M^n$ is an equator.
\end{prop}
\begin{proof}
Since $$\Delta\varphi_a=-n\varphi_a,$$
if $\varphi_a\geq 0$ for some $a\in \mathbb{S}^{n+1}$, it implies that
$$\Delta\varphi_a\leq0.$$
But $$\int_{M}\Delta\varphi_a=0,$$
one has $\Delta\varphi_a\equiv 0=\varphi_a$ and thus $ M^n$ is totally geodesic.
%This completes the proof.
\end{proof}

\begin{proof}[\textbf{Proof of Theorem $\mathbf{\ref{thm  applitions of intergal formula}}$}]
Case $({\rm i})$.
If $M^n$ lies in some spherical zone $\mathbb{S}_{zone}^{n+1}
(\sqrt{\frac{1}{2n}})$  completely,  then there is some  $a_0\in \mathbb{S}^{n}$ such that
$$|\varphi_{a_0}(x)|=|\left\langle x,a_0 \right\rangle| <\sqrt{\frac{1}{2n}},  $$
for all $x\in M^n$.
Then it follows from Theorem \ref{thm main Volume estimation of minimal hypersurface}  the following contradiction
$$
\frac{1}{2n}\leq
\inf_{a\in \mathbb{S}^{n+1}}
\frac {\int_{M} \varphi^2_a}
{  {\rm Vol }\left( M^n \right)  }
<\frac{1}{2n}.
$$
Similarly, for IE minimal CSC hypersurfaces with $|\varphi_{a_0}(x)|<\sqrt{\frac{1}{n+2}}$, we have
$$
\inf_{a\in \mathbb{S}^{n+1}}
\frac {\int_{M} \varphi^2_a}
{  {\rm Vol }\left( M^n \right)  }
<\frac{1}{n+2},
$$
which shows that $M^n$ is totally geodesic by Theorem \ref{thm main Volume estimation of minimal hypersurface}.

Applying Theorem \ref{thm introduction Volume estimation of minimal hypersurface for nonconstant S} and Corollary \ref{cor n in 2-6 Volume estimation of minimal hypersurface for nonconstant S}, cases $({\rm ii})$ and $({\rm iii})$ can be proven similarly as for case $({\rm i})$.
\end{proof}

\begin{acknow}
The authors thank the anonymous referee for their valuable suggestions.
The authors would like to Dr. Qichao Li for his valuable discussions about Takahashi's theorem. Finally, the authors want to  thank Professor Xin Zhou for his useful discussions about minimal surfaces in $3$-sphere.
\end{acknow}

%\section*{Declarations}
%\subsection*{Data availability} No data has been generated or analysed during this study.
%\subsection*{Conflict of interest} On behalf of all authors, the corresponding author states that there is no conflict of interest.

\end{document}